\documentclass[12 pt,draft]{article}
\title{\vspace{-35pt}Pseudo-Anosov dilatations and the Johnson filtration}
\author{Justin Malestein\footnote{Supported in part by ERC grant agreement no 226135}\ \ and Andrew Putman\footnote{Supported in part by NSF grant DMS-1255350 and the Alfred P.\ Sloan Foundation}}
\date{February 2, 2015}

\usepackage{amsmath,amssymb,amsthm,amscd,amsfonts}
\usepackage{epsfig,pinlabel}
\usepackage[vmargin=1in, hmargin=1.25in]{geometry}
\usepackage[font=small,format=plain,labelfont=bf,up,textfont=it,up]{caption}
\usepackage{stmaryrd}
\usepackage{verbatim}
\usepackage{type1cm}
\usepackage{calc}
\usepackage[all,cmtip]{xy}

\theoremstyle{plain}
\newtheorem{theorem}{Theorem}[section]
\newtheorem{maintheorem}{Theorem}

\newtheorem{lemma}[theorem]{Lemma}

\newtheorem{step}{Step}

\newcommand\BeginSteps{\setcounter{step}{0}}

\theoremstyle{definition}

\theoremstyle{remark}
\newtheorem*{remark}{Remark}

\DeclareMathOperator{\Mod}{Mod}

\newcommand\R{\ensuremath{\mathbb{R}}}
\newcommand\C{\ensuremath{\mathbb{C}}}
\newcommand\Z{\ensuremath{\mathbb{Z}}}

\newcommand\Field{\ensuremath{\mathbb{F}}}

\DeclareMathOperator{\HH}{H}

\DeclareMathOperator{\Out}{Out}

\newcommand\CaptionSpace{\hspace{0.2in}}

\newcommand\Set[2]{\ensuremath{\{\text{#1 $|$ #2}\}}}

\newcommand\Figure[3]{
\begin{figure}[t]
\centering
\centerline{\psfig{file=#2,scale=55}}
\caption{#3}
\label{#1}
\end{figure}}

\newcommand\FigureB[3]{
\begin{figure}[b]
\centering
\centerline{\psfig{file=#2,scale=55}}
\caption{#3}
\label{#1}
\end{figure}}

\newcommand\Johnson{\ensuremath{\mathcal{N}}}

\begin{document}

\maketitle

\vspace{-25pt}
\begin{abstract}
Answering a question of Farb--Leininger--Margalit, we give explicit lower bounds for the dilatations of pseudo-Anosov mapping classes
lying in the $k^{\text{th}}$ term of the Johnson filtration of the mapping class group.
\end{abstract}

\section{Introduction}

Let $\Mod(\Sigma)$ be the mapping class group of a closed orientable surface $\Sigma$.
Thurston's well-known classification of surface homeomorphisms (see \cite{FLP}) says that every element of $\Mod(\Sigma)$ is either finite order, reducible,
or pseudo-Anosov.  By definition, a pseudo-Anosov mapping class
$f \in \Mod(\Sigma)$ is one that can be represented by a homeomorphism $F : \Sigma \rightarrow \Sigma$ such that
there exist two transverse singular measured foliations
$\mathcal{F}^{\text{u}}$ and $\mathcal{F}^{\text{s}}$ and some $\lambda(f) > 1$ such that $F_{\ast}(\mathcal{F}^{\text{u}}) = \lambda(f) \cdot \mathcal{F}^{\text{u}}$
and $F_{\ast}(\mathcal{F}^{\text{s}}) = \frac{1}{\lambda(f)} \cdot \mathcal{F}^{\text{s}}$.  The number $\lambda(f)$ only depends
on $f$ and is known as the {\em dilatation} of $f$.  The dilatation $\lambda(f)$ of a pseudo-Anosov
mapping class $f$ shows up in many places; for instance, the number $\ln(\lambda(f))$ is the translation length
with respect to the Teichm\"{u}ller metric of the action of $f$ on Teichm\"{u}ller space.

\paragraph{Minimal dilatations.}
The set of possible dilatations of pseudo-Anosov mapping classes has many interesting properties.  Define
$$\text{Spec}(\Mod(\Sigma)) = \Set{$\ln(\lambda(f))$}{$f \in \Mod(\Sigma)$ is pseudo-Anosov} \subset \R.$$
Arnoux--Yoccoz \cite{ArnouxYoccoz} and Ivanov \cite{Ivanov} independently proved that $\text{Spec}(\Mod(\Sigma))$ is closed
and discrete, so it has a minimal element $L(\Mod(\Sigma))$.  Penner \cite{PennerDil} proved that there exists some $C > 1$
such that if $\Sigma$ has genus $g$, then
$$\frac{1}{C g} \leq L(\Mod(\Sigma)) \leq \frac{C}{g}.$$
In particular, $L(\Mod(\Sigma)) \rightarrow 0$ as $g \rightarrow \infty$.

\paragraph{Johnson filtration.}
The {\em Johnson filtration} is an important sequence of subgroups of $\Mod(\Sigma)$.  
Recall that if $G$ is a group, then the {\em lower central series} of $G$ is the inductively defined sequence
$$\gamma_1(G) = G  \quad \text{and} \quad \gamma_{d+1}(G) = [G,\gamma_d(G)] \text{ for } d \geq 1.$$
Letting $\pi_1 = \pi_1(\Sigma)$, the $k^{\text{th}}$ term of the Johnson filtration of $\Mod(\Sigma)$, denoted
$\Johnson_k(\Sigma)$, is the kernel of the natural representation $\Mod(\Sigma) \rightarrow \Out(\pi_1 / \gamma_{k+1}(\pi_1))$.
For example, since $\pi_1 / \gamma_2(\pi_1) \cong \HH_1(\Sigma;\Z)$, the group $\Johnson_1(\Sigma)$ is the 
{\em Torelli group}, that is, the kernel of the action of $\Mod(\Sigma)$ on $\HH_1(\Sigma;\Z)$.
The groups $\Johnson_k(\Sigma)$ were
first defined in \cite{JohnsonSurvey} and have connections to number theory (see \cite{Matsumoto})
and $3$-manifolds (see \cite{GaroufalidisLevine}); however, many basic questions about them remain open.

\paragraph{Dilatation in the Johnson filtration.}
If $H < \Mod(\Sigma)$ is a subgroup, then define
$$\text{Spec}(H) = \Set{$\ln(\lambda(f))$}{$f \in H$ is pseudo-Anosov} \subset \R_{>0}.$$
Possibly $\text{Spec}(H)$ is empty; otherwise, it must have a least element which we will denote by $L(H)$.
Farb--Leininger--Margalit \cite{FarbLeiningerMargalit} proved that if $\Sigma$ is a closed surface whose genus is at least $2$ and $k \geq 1$, then
$\text{Spec}(\Johnson_k(\Sigma))$ is nonempty.  They also proved that there exist numbers 
$n_k>0$ which depend only on $k$ such that $L(\Johnson_k(\Sigma)) \geq n_k$.
Observe that this contrasts sharply with Penner's theorem
for the whole mapping class group.  Finally, they proved that $n_k \rightarrow \infty$ as $k \rightarrow \infty$.

\paragraph{Main theorem.}
Farb--Leininger--Margalit's proof that $n_k \rightarrow \infty$ as $k \rightarrow \infty$ relies on a sort of
compactness argument and gives no estimates for $n_k$.  They posed the question of obtaining explicit estimates.
Our main theorem answers their question.  It says that $L(\Johnson_k(\Sigma))$ is bounded below by a quantity that
grows roughly like $\ln(k)$.  More precisely, we have the following.

\begin{maintheorem}
\label{maintheorem:dil}
Let $\Sigma$ be a closed surface whose genus is at least $2$.  Then for all $k \geq 1$ we have
$$L\left(\Johnson_k\left(\Sigma\right)\right) > \max\left(0.197, c \cdot \ln\left(\frac{k+3}{2}\right) - \ln\left(2\right)\right),$$
where $c = \frac{\ln(28/25)}{\ln(4)}$.
\end{maintheorem}

\begin{remark}
In \cite[\S 4]{FarbLeiningerMargalit}, Farb--Leininger--Margalit prove that $L(\Johnson_k(\Sigma))$ is bounded above
by a function which is exponential in $k$ (independent of the genus).  Letting $F$ be a free group
of rank $2$, the key input to their construction is a sequence of nontrivial
elements $w_k \in \gamma_k(F)$ whose
word lengths are exponential in $k$.  In fact, there are nontrivial elements of $\gamma_k(F)$ whose word
lengths are quadratic in $k$ (see the remarks after \cite[Theorem 2]{MalesteinPutman}).  Plugging
these into their construction, one can deduce that $L(\Johnson_k(\Sigma))$ is bounded above
by a function which is quadratic in $k$ (independent of the genus).
\end{remark}

\paragraph{Curves on a surface.}
If $\alpha$ and $\beta$ are simple closed curves on $\Sigma$, then denote by $i(\alpha,\beta)$ the {\em geometric
intersection number} of $\alpha$ and $\beta$, that is, the minimum size of $\alpha' \cap \beta'$ as $\alpha'$ and $\beta'$
range over curves homotopic to $\alpha$ and $\beta$, respectively.  One of the key insights of \cite{FarbLeiningerMargalit} 
is that one can obtain estimates on $\lambda(f)$ by bounding the size of $i(f(\delta),\delta)$ from below for all nonnullhomotopic
simple closed curves $\delta$.  Indeed, \cite[Proposition 2.7]{FarbLeiningerMargalit} says that if $f \in \Mod(\Sigma)$
is pseudo-Anosov and if $i(f(\delta),\delta) \geq n \geq 3$ for all nonnullhomotopic simple closed curves $\delta$, then
$\lambda(f) > n/2$.  

Recall that there is a bijection between free homotopy classes of oriented closed curves on $\Sigma$ and conjugacy classes
in $\pi_1$.  If $f \in \Johnson_k(\Sigma)$ and $\delta$ is an oriented simple closed curve, then by definition
the conjugacy classes in $\pi_1$ associated to $f(\delta)$ and $\delta$ project to the same conjugacy class
in $\pi_1 / \gamma_{k+1}(\pi_1)$.  The following theorem (applied with $d = k+1$) therefore provides a lower bound on
$i(f(\delta),\delta)$.

\begin{maintheorem}
\label{maintheorem:isect}
Let $\Sigma$ be a closed surface whose genus is at least $2$ and let
$\alpha$ and $\beta$ be nonisotopic oriented simple closed curves on $\Sigma$.  Assume
that for some $d \geq 3$ the conjugacy classes in $\pi_1(\Sigma) / \gamma_d(\pi_1(\Sigma))$
associated to $\alpha$ and $\beta$ are the same.  Then
$$i(\alpha,\beta) \geq \left(\frac{d+2}{2}\right)^c,$$
where $c = \frac{\ln(28/25)}{\ln(4)}$.
\end{maintheorem}

\begin{remark}
In Appendix \ref{appendix:nonclosed}, we show that the conclusion of Theorem \ref{maintheorem:isect} also holds
for compact surfaces with boundary.
\end{remark}

\paragraph{Deriving Theorem \ref{maintheorem:dil} from Theorem \ref{maintheorem:isect}.}
We now discuss how to derive Theorem \ref{maintheorem:dil} from Theorem \ref{maintheorem:isect}.
The bound in Theorem \ref{maintheorem:isect} is greater than $2$ starting at $d=9623$.  
Therefore, Theorem \ref{maintheorem:isect} together with the aforementioned result \cite[Proposition 2.7]{FarbLeiningerMargalit} implies
that for $k \geq 9622$ we have 
\begin{equation}
\label{eqn:bound}
L\left(\Johnson_k\left(\Sigma\right)\right) \geq \ln\left(\frac{1}{2} \left(\frac{k+3}{2}\right)^c\right) = c \cdot \ln\left(\left(\frac{k+3}{2}\right)\right) - \ln\left(2\right),
\end{equation}
where $c = \frac{\ln(28/25)}{\ln(4)}$.  Farb--Leininger--Margalit also proved that $L(\Johnson_1(\Sigma)) > 0.197$ (see
\cite[Theorem 1.1]{FarbLeiningerMargalit}), which clearly implies that $L(\Johnson_k(\Sigma)) > 0.197$ for all $k \geq 1$.  The bound
in \eqref{eqn:bound} is less than $0.197$ for $k < 9622$, so Theorem \ref{maintheorem:dil} follows.

\paragraph{Proof techniques for Theorem \ref{maintheorem:isect}.}
The heart of our proof of Theorem \ref{maintheorem:isect} is the construction of a suitable nilpotent
cover of $\Sigma$ which ``resolves'' the intersections of $\alpha$ and $\beta$.  This
nilpotent cover is obtained by constructing a sequence of $2$-fold covers each of which resolves at minimum some constant fraction of
of the intersections.  It is easy to construct covers that resolve a single intersection at a time, but that only leads
to a logarithmic bound in Theorem \ref{maintheorem:isect}.  To get our stronger bound, we need to resolve many intersections
at once.  It seems difficult to explicitly construct these $2$-fold covers, but we show that they exist
using probabilistic arguments.  Namely, we prove that if the covers are chosen at random in an appropriate way, then
the expected value of the number of resolved intersections is larger than some constant fraction of the intersections, and thus
that there must exist {\em some} cover that resolves enough intersections.  See the proofs of the
key Lemmas \ref{lemma:basicmoveresolveisect}--\ref{lemma:basicmovemakegood} below.

\paragraph{Subtleties.}
We close this introduction by discussing some difficulties that arise in proving Theorem \ref{maintheorem:isect}.  Though
in the end we will not phrase it that way, 
one can view our proof as deriving a contradiction from the existence of nonisotopic simple closed curves $\alpha$ and $\beta$ on $\Sigma$ that induce
the same conjugacy classes in $\pi_1(\Sigma) / \gamma_d(\pi_1(\Sigma))$ and that satisfy $i(\alpha,\beta) < \left(\frac{d+2}{2}\right)^c$.
One difficulty that must be overcome is that given two such curves $\alpha$ and $\beta$ and a basepoint
$v \in \Sigma$, it seems hard to find 
based representatives $\widehat{\alpha},\widehat{\beta} \in \pi_1(\Sigma,v)$ of $\alpha$
and $\beta$ with the following two properties.
\begin{itemize}
\item $\widehat{\alpha}$ and $\widehat{\beta}$ intersect $i(\alpha,\beta)$ times (perhaps up to a constant factor), and
\item $\widehat{\alpha} \widehat{\beta}^{-1} \in \gamma_d(\pi_1(\Sigma,v))$.
\end{itemize}
In other words, the algebraic and topological conditions on $\alpha$ and $\beta$ do not interact very well.

If such $\widehat{\alpha}$ and $\widehat{\beta}$ existed, then the self-intersection number 
of $\widehat{\alpha} \widehat{\beta}^{-1}$ would be at most $i(\alpha,\beta)$ (up to a constant factor) and one could appeal to the paper \cite{MalesteinPutman} of the authors,
which bounds from below the self-intersection number of nontrivial elements of $\gamma_d(\pi_1(\Sigma,v))$.  Our proof does share
some ideas with \cite{MalesteinPutman}, but substantial new ideas were needed
to overcome the difficulties just discussed.  

\begin{remark}
We do not know any examples of curves $\alpha$ and $\beta$ as above that
cannot be realized by based curves with the above properties, but we conjecture that they exist.
\end{remark}

\paragraph{Acknowledgments.}
We want to thank Tom Church for helpful comments.


\section{Initial reduction}
\label{section:reduction}

We will prove Theorem \ref{maintheorem:isect} by constructing a certain finite cover of $\Sigma$
that ``resolves'' the intersections of $\alpha$ and $\beta$ in an appropriate way.  In this section,
we will state the lemma that gives this cover and then show how to derive Theorem \ref{maintheorem:isect}.

\paragraph{Curve-arc triples.}
We begin with some necessary definitions.  Fix a closed surface $\Sigma$ whose genus is at least $2$,
and choose a hyperbolic metric on $\Sigma$.  A {\em curve-arc triple} on $\Sigma$
is a tuple $((\alpha,v),\tau,(\beta,w))$ as follows.
\begin{itemize}
\item The element $\alpha$ is a simple closed oriented geodesic on $\Sigma$ which is based at $v \in \Sigma$.
\item The element $\beta$ is a simple closed oriented geodesic on $\Sigma$ which is based at $w \in \Sigma$.
\item $\alpha$ and $\beta$ are nonisotopic.
\item The element $\tau$ is a path (possibly with self-intersections) from $v$ to $w$.
\end{itemize}
If $((\alpha,v),\tau,(\beta,w))$ is a curve-arc triple and $(\widetilde{\Sigma},\widetilde{v}) \rightarrow (\Sigma,v)$ is
a (based) cover, then we can {\em lift} the triple $((\alpha,v),\tau,(\beta,w))$ to $(\widetilde{\Sigma},\widetilde{v})$
as follows.  First, endow $\widetilde{\Sigma}$ with the hyperbolic metric that makes the covering map a local isometry.  The
arc $\tau$ lifts to an arc $\widetilde{\tau}$ starting at $\widetilde{v}$; let $\widetilde{w}$ be the endpoint of $\widetilde{\tau}$.
Then $\alpha$ lifts to a geodesic arc $\widetilde{\alpha}$ starting at $\widetilde{v}$ and $\beta$ lifts to a geodesic arc
$\widetilde{\beta}$ starting at $\widetilde{w}$.  We will say that $((\widetilde{\alpha},\widetilde{v}),\widetilde{\tau},(\widetilde{\beta},\widetilde{w}))$ is
a {\em closed lift} of $((\alpha,v),\tau,(\beta,w))$ if $\widetilde{\alpha}$ and $\widetilde{\beta}$ are simple closed curves (and
thus $((\widetilde{\alpha},\widetilde{v}),\widetilde{\tau},(\widetilde{\beta},\widetilde{w}))$ is a curve-arc triple on $\widetilde{\Sigma}$).  We
will say that $((\alpha,v),\tau,(\beta,w))$ has {\em only a partially closed lift} if one of $\widetilde{\alpha}$ and
$\widetilde{\beta}$ is closed and the other is not closed.  Finally, we will say
that $((\alpha,v),\tau,(\beta,w))$ has a {\em nonclosed lift} if neither $\widetilde{\alpha}$ nor
$\widetilde{\beta}$ is closed.

\begin{remark}
The reason we require our curves to be geodesics with respect to some hyperbolic metric is that this implies that
they intersect minimally (see \cite[Corollary 1.9]{FarbMargalitPrimer}).  Moreover, this persists when we pass
to finite covers.
\end{remark}

\paragraph{The key lemma.}
With these definitions in hand, we can state the key lemma of this paper.

\begin{lemma}
\label{lemma:resolveallisect}
Let $\Sigma$ be a closed surface whose genus is at least $2$.  Fix a hyperbolic metric
on $\Sigma$ and let $((\alpha,v),\tau,(\beta,w))$ be a curve-arc triple on
$\Sigma$.  Set $n = i(\alpha,\beta)$.  Then for some $k$ satisfying
$$k \leq \begin{cases}
2 \log_{28/25}(n)+1 & \text{if $n \geq 2$}\\
3 & \text{if $0 \leq n \leq 1$}\end{cases}$$ 
there exists a tower
$$(\Sigma_k,v_k) \longrightarrow (\Sigma_{k-1},v_{k-1}) \longrightarrow \cdots \longrightarrow (\Sigma_0,v_0) = (\Sigma,v)$$
of regular degree $2$ based covers such that $((\alpha,v),\tau,(\beta,w))$ has
only a partially closed lift to $(\Sigma_k,v_k)$.
\end{lemma}

\begin{remark}
It is important that the lift is partially closed -- the proof below will not work
if it is either closed or nonclosed.
\end{remark}

\noindent
The proof of Lemma \ref{lemma:resolveallisect} is contained in \S \ref{section:constructcover}.

\paragraph{The main theorem.}
Before deducing Theorem \ref{maintheorem:isect} from Lemma \ref{lemma:resolveallisect},
we need the following two known lemmas. Recall that a group $G$ is {\it at most $k$-step nilpotent}
if $\gamma_{k+1}(G) = 0$.

\begin{lemma}[{\cite[\S 6.1 Proposition 2]{DummitFoote}}]
\label{lemma:detectnilpotent}
Let $G$ be a finite group of order $2^k$ for some $k \geq 2$.  Then
$G$ is at most $(k-1)$-step nilpotent, and hence $\gamma_k(G) = 0$.
\end{lemma}

\begin{lemma}[{\cite[Lemma 2.3]{MalesteinPutman}}]
\label{lemma:passtonormal}
Let
$$\Gamma_k \lhd \Gamma_{k-1} \lhd \cdots \lhd \Gamma_0$$
be groups satisfying $[\Gamma_i:\Gamma_{i+1}] = 2$ for $0 \leq i < k$.  Then there
exists a subgroup $\Gamma' \lhd \Gamma_k$ such that $\Gamma' \lhd \Gamma_0$
and $[\Gamma_0:\Gamma'] = 2^{\ell}$ with
$\ell \leq 2^k-1$.
\end{lemma}

\begin{proof}[{Proof of Theorem \ref{maintheorem:isect}}]
As in the statement of the theorem, let $\Sigma$ be a closed surface
whose genus is at least $2$ and let
$\alpha$ and $\beta$ be nonisotopic oriented simple closed curves on $\Sigma$ such that
for some $d \geq 3$ the conjugacy classes in $\pi_1(\Sigma) / \gamma_d(\pi_1(\Sigma))$
associated to $\alpha$ and $\beta$ are the same.  Set $n = i(\alpha,\beta)$.  Our goal is to show
that 
$$n \geq \left(\frac{d+2}{2}\right)^c$$
with $c = \frac{\ln(28/25)}{\ln(4)}$.

We first claim that $n \geq 2$.  If $n=1$, then $\alpha$ and $\beta$ have algebraic intersection number $\pm 1$, which is impossible
since they are homologous.  Assume now that $n=0$, i.e.\ that $\alpha$ and $\beta$ are disjoint.  We divide this into two cases.
\begin{itemize}
\item If $\alpha$ and $\beta$ are nonseparating, then \cite[Theorem 1.1.2]{KawazumiKuno} implies that $T_{\alpha} T_{\beta}^{-1} \in \Johnson_{d-1}(\Sigma)$.
However, $T_{\alpha} T_{\beta}^{-1} \notin \Johnson_2(\Sigma) \supset \Johnson_{d-1}(\Sigma)$; see \cite{JohnsonAbel}.
\item If $\alpha$ and $\beta$ are separating, then \cite[Theorem 1.1.2]{KawazumiKuno} implies that $T_{\alpha} T_{\beta}^{-1} \in \Johnson_d(\Sigma)$.  
However, $T_{\alpha} T_{\beta}^{-1} \notin \Johnson_3 \supset \Johnson_d(\Sigma)$; see \cite[Appendix A]{BestvinaBuxMargalit}.
\end{itemize}

\begin{remark}
As stated, \cite[Theorem 1.1.2]{KawazumiKuno} concerns surfaces with one boundary component.  However, the desired
result for closed surfaces easily follows from this via the map on mapping class groups obtained by gluing a disc to the
boundary component.
\end{remark}

Fix a hyperbolic metric on $\Sigma$ and isotope $\alpha$ and $\beta$ such that they are geodesics.
Fix basepoints $v$ and $w$
for $\alpha$ and $\beta$, respectively, and regard them as based curves.  Choose an arc
$\tau'$ from $v$ to $w$.  By assumption, $\alpha \in \pi_1(\Sigma,v)$ and 
$\tau' \cdot \beta \cdot (\tau')^{-1} \in \pi_1(\Sigma,v)$ map to conjugate elements
in $\pi_1(\Sigma,v) / \gamma_d(\pi_1(\Sigma,v))$.  There thus exists some
$\tau'' \in \pi_1(\Sigma,v)$ such that letting $\tau = \tau'' \cdot \tau'$, we have
$\tau \cdot \beta^{-1} \cdot \tau^{-1} \cdot \alpha \in \gamma_d(\pi_1(\Sigma,v))$.  The
triple $((\alpha,v),\tau,(\beta,w))$ is a curve-arc triple.

By Lemma \ref{lemma:resolveallisect}, for some $k \leq 2 \log_{28/25}(n)+1$, there exists a tower
$$(\Sigma_k,v_k) \longrightarrow (\Sigma_{k-1},v_{k-1}) \longrightarrow \cdots \longrightarrow (\Sigma_0,v_0) = (\Sigma,v)$$
of regular degree $2$ based covers such that $((\alpha,v),\tau,(\beta,w))$ has only a partially closed 
lift to $(\Sigma_k,v_k)$.  For $0 \leq i \leq k$, let $\Gamma_i \leq \pi_1(\Sigma,v)$ be the subgroup associated
to $(\Sigma_i,v_i)$, so we have a sequence
$$\Gamma_k \lhd \Gamma_{k-1} \lhd \cdots \lhd \Gamma_0$$
of groups.  Lemma \ref{lemma:passtonormal} says that there exists a subgroup $\Gamma' \lhd \Gamma_k$
such that $\Gamma' \lhd \Gamma_0$ and $[\Gamma_0:\Gamma'] = 2^{\ell}$ with $\ell \leq 2^k-1$.
Since $((\alpha,v),\tau,(\beta,w))$ has only a partially closed lift to $(\Sigma_k,v_k)$,
one element in $\{\alpha,\tau \cdot \beta \cdot \tau^{-1}\}$ lies in $\Gamma_k$ and the
other does not lie in $\Gamma_k$.  We deduce that 
$\tau \cdot \beta^{-1} \cdot \tau^{-1} \cdot \alpha \notin \Gamma_k$, and hence
$\tau \cdot \beta^{-1} \cdot \tau^{-1} \cdot \alpha \notin \Gamma'$.

Lemma \ref{lemma:detectnilpotent} implies that $\Gamma_0 / \Gamma'$ is
at most $(\ell-1)$-step nilpotent, so $\gamma_{\ell}(\Gamma_0) \subset \Gamma'$.  Since
$\tau \cdot \beta^{-1} \cdot \tau^{-1} \cdot \alpha \notin \Gamma'$ and
$\tau \cdot \beta^{-1} \cdot \tau^{-1} \cdot \alpha \in \gamma_d(\Gamma_0)$, we
conclude that $d < \ell$, i.e.\ that 
$$d \leq 2^k - 2 \leq 2^{2 \log_{28/25}(n)+1} - 2.$$
A little elementary algebra then shows that
$$n \geq \left(\frac{d+2}{2}\right)^c$$
with $c = \frac{\ln(28/25)}{\ln(4)}$, as desired.
\end{proof}


\section{Building a tower}
\label{section:constructcover}

This section is devoted to the proof of Lemma \ref{lemma:resolveallisect}. 
We begin by presenting several basic constructions used in the proof in \S \ref{section:moves}.
We then prove the lemma in \S \ref{section:constructcoverproof}.

\subsection{Basic constructions}
\label{section:moves}

To simplify our notation, from this point onward we will write $\Field_2$ instead
of $\Z/2$.  For an oriented closed curve $\delta$ in $\Sigma$, 
the associated element of $\HH_1(\Sigma;\Field_2)$ is denoted by $[\delta]$.

\begin{lemma}
\label{lemma:different}
Let $\Sigma$ be a closed surface whose genus is at least $2$.  Fix a hyperbolic metric
on $\Sigma$ and let $((\alpha,v),\tau,(\beta,w))$ be a curve-arc triple on
$\Sigma$.  Assume that $[\alpha] \neq [\beta]$.  Then there exists a regular degree
$2$ cover $(\widetilde{\Sigma},\widetilde{v}) \rightarrow (\Sigma,v)$ such that
$((\alpha,v),\tau,(\beta,w))$ has only a partially closed lift to $(\widetilde{\Sigma},\widetilde{v})$.
\end{lemma}
\begin{proof}
Pick a homomorphism $\phi : \HH_1(\Sigma;\Field_2) \rightarrow \Field_2$ 
such that $\phi([\alpha]) \neq \phi([\beta])$, and
let $(\widetilde{\Sigma},\widetilde{v}) \rightarrow (\Sigma,v)$ be the $2$-fold regular cover associated to
$\phi$.  Let the lift of $((\alpha,v),\tau,(\beta,w))$ to
$(\widetilde{\Sigma},\widetilde{v})$ be
$((\widetilde{\alpha},\widetilde{v}),\widetilde{\tau},(\widetilde{\beta},\widetilde{w}))$.  Clearly
$\widetilde{\alpha}$ is closed (resp.\ $\widetilde{\beta}$ is closed) if and only if
$\phi([\widetilde{\alpha}]) = 0$
(resp.\ $\phi([\widetilde{\beta}]) = 0$).
Since $\phi([\alpha]) \neq \phi([\beta])$,
we deduce that $((\alpha,v),\tau,(\beta,w))$ has only a partially closed lift
to $(\widetilde{\Sigma},\widetilde{v})$, as desired.
\end{proof}

\noindent
Next, we need the following.

\FigureB{figure:liftnosep}{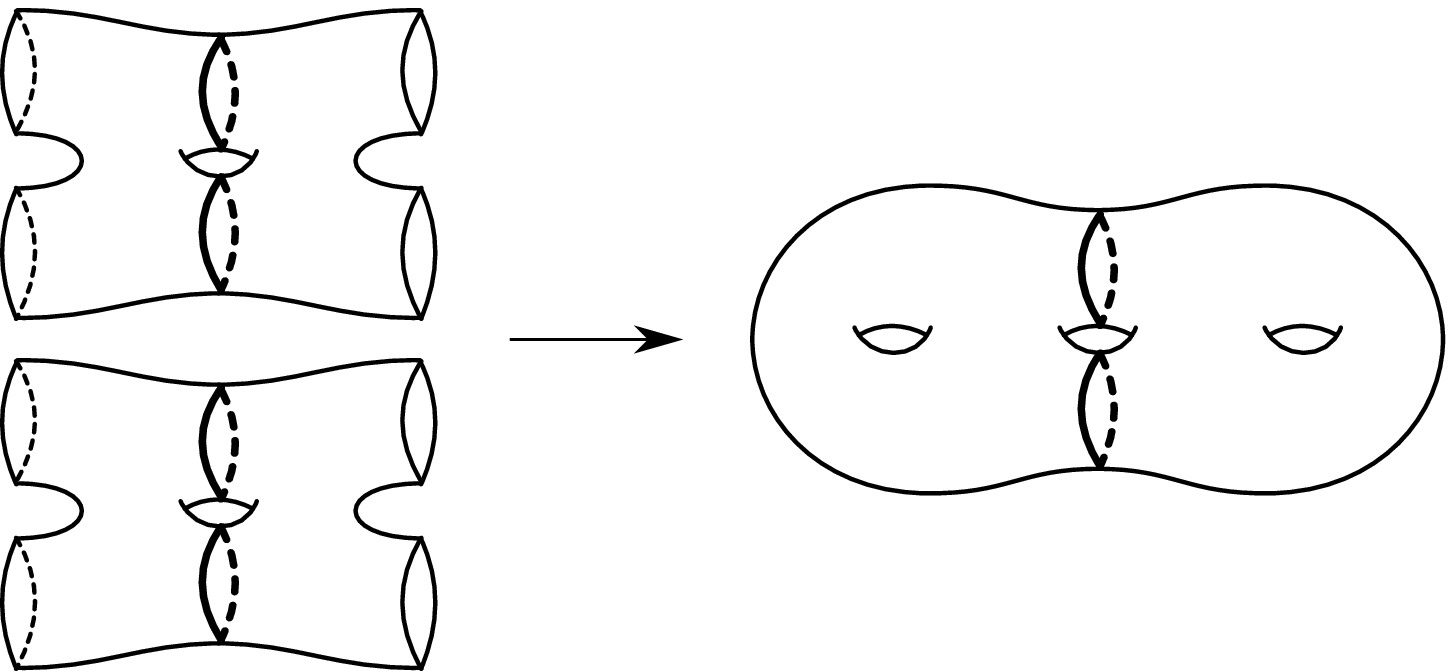}{In the left hand surface, the boundary components are glued in pairs to form a closed connected surface $\widetilde{\Sigma}$.
This has a regular degree $2$ covering map to the right hand surface $\Sigma$; the deck group exchanges the top and bottom piece while flipping them.  The
preimage of $\alpha$ is $\widetilde{\alpha}_1 \sqcup \widetilde{\alpha}_2$ and the preimage of $\beta$ is $\widetilde{\beta}_1 \sqcup \widetilde{\beta}_2$.}

\begin{lemma}
\label{lemma:bp}
Let $\Sigma$ be a closed surface whose genus is at least $2$.  Fix a hyperbolic metric
on $\Sigma$ and let $((\alpha,v),\tau,(\beta,w))$ be a curve-arc triple on
$\Sigma$.  Assume that $\alpha$ and $\beta$ are disjoint and that $[\alpha] = [\beta] \neq 0$.
Then there exists a regular degree $2$ cover
$(\widetilde{\Sigma},\widetilde{v}) \rightarrow (\Sigma,v)$ such that $((\alpha,v),\tau,(\beta,w))$
has a closed lift $((\widetilde{\alpha},\widetilde{v}),\widetilde{\tau},(\widetilde{\beta},\widetilde{w}))$ to
$(\widetilde{\Sigma},\widetilde{v})$ satisfying $[\widetilde{\alpha}] \neq [\widetilde{\beta}]$.
\end{lemma}
\begin{proof}
The curves $\alpha$ and $\beta$ are disjoint and homologous over $\Field_2$.
It is easy to see that this implies
that they are actually homologous over $\Z$, i.e.\ that they bound an embedded subsurface of $\Sigma$.
As is shown in Figure \ref{figure:liftnosep}, there exists a regular degree $2$ cover $(\widetilde{\Sigma},\widetilde{v}) \rightarrow (\Sigma,v)$
with the following properties.
\begin{itemize}
\item The preimage of $\alpha$ in $\widetilde{\Sigma}$ consists of two disjoint simple closed curves, and similarly
for $\beta$.
\item If $\widetilde{\alpha}$ and $\widetilde{\beta}$ are any components of the preimage in
$\widetilde{\Sigma}$ of $\alpha$ and $\beta$, respectively, then $[\widetilde{\alpha}] \neq [\widetilde{\beta}]$.
\end{itemize}
Clearly $(\widetilde{\Sigma},\widetilde{v}) \rightarrow (\Sigma,v)$ is the desired cover.
\end{proof}

\noindent
A similar idea will yield the following.

\begin{lemma}
\label{lemma:sep}
Let $\Sigma$ be a closed surface whose genus is at least $2$.  Let $\delta$ be a nonnullhomotopic
oriented simple closed curve on $\Sigma$ such that $[\delta] = 0$.  Then there exists a regular degree $2$ cover
$\widetilde{\Sigma} \rightarrow \Sigma$ such that the preimage of $\delta$ in $\widetilde{\Sigma}$ has two components
$\widetilde{\delta}_1$ and $\widetilde{\delta}_2$ satisfying $[\widetilde{\delta}_i] \neq 0$ for $i=1,2$.
\end{lemma}
\begin{proof}
Using the fact that $\delta$ is a simple closed curve which is nullhomologous over $\Field_2$,
it is easy to see that $\delta$ is actually nullhomologous over $\Z$, i.e.\ that $\delta$
separates $\Sigma$.  The needed cover is as depicted in Figure \ref{figure:liftsep}.
\end{proof}

\Figure{figure:liftsep}{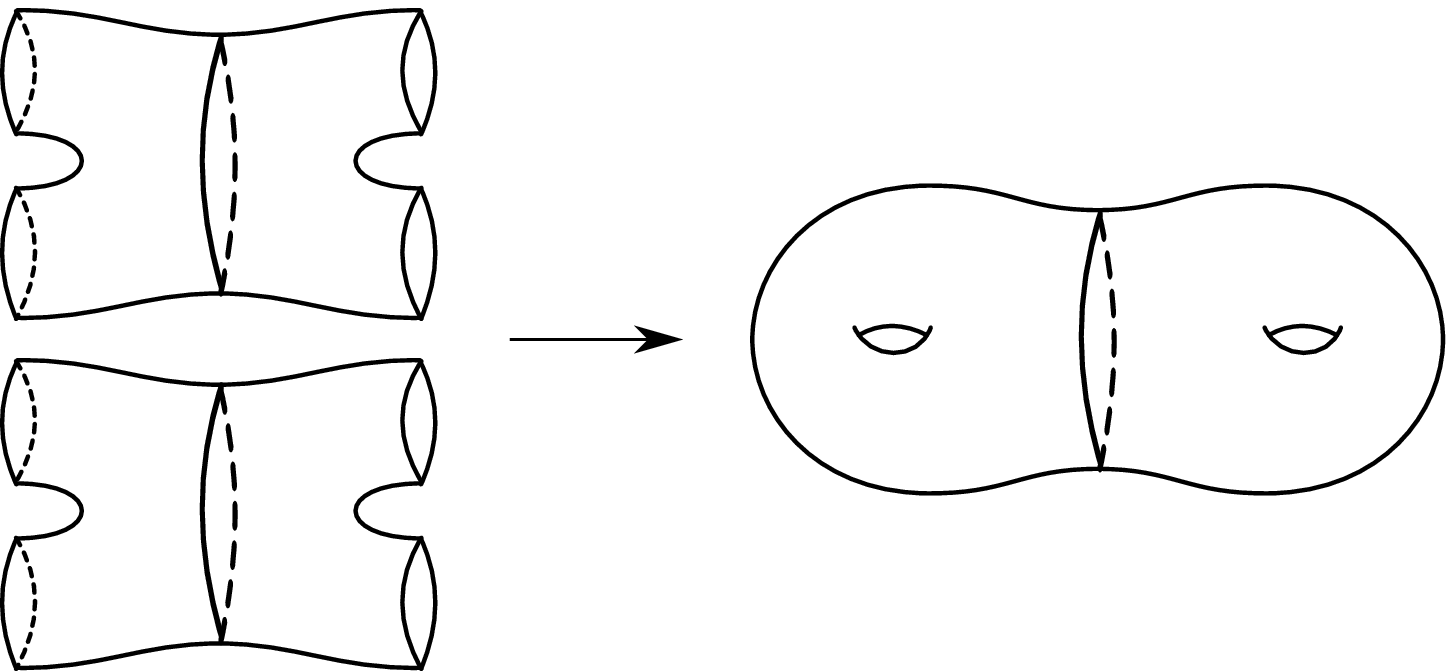}{In the left hand surface, the boundary components are glued in pairs to form a closed connected surface $\widetilde{\Sigma}$.
This has a regular degree $2$ covering map to the right hand surface $\Sigma$; the deck group exchanges the top and bottom piece while flipping them.  The
preimage of $\delta$ is $\widetilde{\delta}_1 \sqcup \widetilde{\delta}_2$.}

\noindent
Finally, the most important construction for the proof of Lemma \ref{lemma:resolveallisect} is the following.

\begin{lemma}
\label{lemma:basicmove}
Let $\Sigma_0$ be a closed surface whose genus is at least $2$.  Fix a hyperbolic metric
on $\Sigma_0$ and let $((\alpha_0,v_0),\tau_0,(\beta_0,w_0))$ be a curve-arc triple on
$\Sigma_0$.  Set $n_0 = i(\alpha_0,\beta_0)$.  Assume that $n_0 \geq 2$ and that neither $[\alpha_0]$
nor $[\beta_0]$ vanishes.  Then for some $q$ satisfying $1 \leq q \leq 2$, there exists
a tower
$$(\Sigma_{q},v_{q}) \longrightarrow \cdots \longrightarrow (\Sigma_0,v_0)$$
of regular degree $2$ based covers such that one of the following holds.
\begin{itemize}
\item The curve-arc triple $((\alpha_0,v_0),\tau_0,(\beta_0,w_0))$ has
only a partially closed lift to $(\Sigma_{q},v_{q})$, or
\item The curve-arc triple $((\alpha_0,v_0),\tau_0,(\beta_0,w_0))$ has a closed
lift $((\alpha_{q},v_{q}),\tau_{q},(\beta_{q},w_{q}))$ to
$(\Sigma_{q},v_{q})$ such that 
$i(\alpha_{q}, \beta_{q}) \leq \frac{25}{28} n_0$.
\end{itemize}
\end{lemma}

\noindent
The proof of Lemma \ref{lemma:basicmove} is difficult and will be postponed until \S \ref{section:basicmove}.

\subsection{The proof of Lemma \ref{lemma:resolveallisect}}
\label{section:constructcoverproof}

All the pieces are now in place for the proof of Lemma \ref{lemma:resolveallisect}.

\begin{proof}[{Proof of Lemma \ref{lemma:resolveallisect}}]
We first recall the statement of Lemma \ref{lemma:resolveallisect}.  Let
$\Sigma$ be a closed surface whose genus is at least $2$.  Fix a hyperbolic metric
on $\Sigma$ and let $((\alpha,v),\tau,(\beta,w))$ be a curve-arc triple on
$\Sigma$.  Set $n = i(\alpha,\beta)$.  Our goal is to show that for some $k$ satisfying
$$k \leq \begin{cases}
2 \log_{28/25}(n)+1 & \text{if $n \geq 2$}\\
3 & \text{if $0 \leq n \leq 1$}\end{cases}$$
there exists a tower
\begin{equation}
\label{eqn:tower}
(\Sigma_k,v_k) \longrightarrow (\Sigma_{k-1},v_{k-1}) \longrightarrow \cdots \longrightarrow (\Sigma_0,v_0) = (\Sigma,v)
\end{equation}
of regular degree $2$ based covers such that $((\alpha,v),\tau,(\beta,w))$ has
only a partially closed lift to $(\Sigma_k,v_k)$. The proof is divided into two steps.

\BeginSteps
\begin{step}
Assume that neither $[\alpha]$ nor $[\beta]$ vanishes.  Then we can find a tower
as in \eqref{eqn:tower} with
$$k \leq \begin{cases}
2 \log_{28/25}(n) & \text{if $n \geq 2$}\\
2 & \text{if $0 \leq n \leq 1$}\end{cases}$$
\end{step}

\noindent
The proof is by induction on $n$.
Set $(\Sigma_0,v_0) = (\Sigma,v)$ 
and $((\alpha_0,v_0),\tau_0,(\beta_0,w_0)) = ((\alpha,v),\tau,(\beta,w))$.

\paragraph{Base cases.}
Suppose $n =0$ or $n=1$.  We divide the verification of these
base cases into three separate cases. 
\begin{itemize}
\item $n=0$ and $[\alpha_0] \neq [\beta_0]$.  Lemma \ref{lemma:different} implies that
there exists a regular degree $2$ cover $(\Sigma_1,v_1) \rightarrow (\Sigma_0,v_0)$
such that $((\alpha_0,v_0),\tau_0,(\beta_0,w_0))$ has only a partially closed lift
to $(\Sigma_1,v_1)$.
\item $n=0$ and $[\alpha] = [\beta]$.  We first apply Lemma \ref{lemma:bp} to obtain
a regular degree $2$ cover $(\Sigma_1,v_1) \rightarrow (\Sigma_0,v_0)$ such that
$((\alpha_0,v_0),\tau_0,(\beta_0,w_0))$ has a closed lift
$((\alpha_1,v_1),\tau_1,(\beta_1,w_1))$ to $(\Sigma_1,v_1)$ satisfying $[\alpha_1] \neq [\beta_1]$.
Lemma \ref{lemma:different} then implies that there is a regular degree $2$ cover $(\Sigma_2,v_2) \rightarrow (\Sigma_1,v_1)$ such that
$((\alpha_1,v_1),\tau_1,(\beta_1,w_1))$ has only a partially closed lift to $(\Sigma_2,v_2)$.
\item $n=1$.  The $\Field_2$-algebraic intersection numbers of $[\alpha_0]$ and
$[\beta_0]$ must be $1$, so in particular we must have $[\alpha_0] \neq [\beta_0]$.  Lemma
\ref{lemma:different} then implies that there exists a regular degree $2$ cover $(\Sigma_1,v_1) \rightarrow (\Sigma_0,v_0)$
such that $((\alpha_0,v_0),\tau_0,(\beta_0,w_0))$ has only a partially closed lift
to $(\Sigma_1,v_1)$.
\end{itemize}

\paragraph{Inductive step.}
Now assume that $n \geq 2$ and that the claim is true for all smaller values of $n$.  Let
$$(\Sigma_{\ell},v_{\ell}) \longrightarrow \cdots \longrightarrow (\Sigma_0,v_0)$$
be the tower of regular degree $2$ covers provided by Lemma \ref{lemma:basicmove}, so $1 \leq \ell \leq 2$.  If
$((\alpha_0,v_0),\tau_0,(\beta_0,w_0))$ has only a partially closed lift to $(\Sigma_{\ell},v_{\ell})$, then
we are done.  Otherwise, let $((\alpha_{\ell},v_{\ell}),\tau_{\ell},(\beta_{\ell},w_{\ell}))$
be the lift of $((\alpha_0,v_0),\tau_0,(\beta_0,w_0))$ to $(\Sigma_{\ell},v_{\ell})$.  Set 
$n_{\ell} = i(\alpha_{\ell},\beta_{\ell})$, so $n_{\ell} \leq \frac{25}{28} n$.  By induction,
for some $k \geq \ell$ satisfying
$$k-\ell \leq \begin{cases}
2 \log_{28/25}(n_{\ell}) & \text{if $n_{\ell} \geq 2$}\\
2 & \text{if $0 \leq n_{\ell} \leq 1$}\end{cases}$$
there exists a tower of regular degree $2$ covers
$$(\Sigma_k,v_k) \longrightarrow \cdots \longrightarrow (\Sigma_{\ell},v_{\ell})$$
such that $((\alpha_{\ell},v_{\ell}),\tau_{\ell},(\beta_{\ell},w_{\ell}))$ has only
a partially closed lift to $(\Sigma_k,v_k)$.  

We claim that
$$(\Sigma_k,v_k) \longrightarrow \cdots \longrightarrow (\Sigma_{0},v_{0})$$
is the desired tower of regular degree $2$ covers.  The only thing that needs
verification is the bound on $k$.  There are two cases.  If $n_{\ell} \leq 1$, then
$k-\ell \leq 2$, and thus
$$k \leq \ell + 2 \leq 2 + 2 \leq 2 \log_{28/25}(2) \leq 2 \log_{28/25}(n),$$
as desired.  Otherwise, $n_{\ell} \geq 2$.  Hence $k-\ell \leq 2 \log_{28/25}(n_{\ell})$ and
$$k \leq \ell + 2 \log_{28/25}\left(n_{\ell}\right) \leq 2 + 2 \log_{28/25}\left(\frac{25}{28} n\right) = 2 \log_{28/25}\left(n\right),$$
as desired.

\begin{step}
Assume that at least one of $[\alpha]$ and $[\beta]$ vanishes.  Then we can find a tower
as in \eqref{eqn:tower} with
$$k \leq \begin{cases}
2 \log_{28/25}(n)+1 & \text{if $n \geq 2$}\\
3 & \text{if $0 \leq n \leq 1$}\end{cases}$$
\end{step}

\noindent
We will give the details for the case where $[\alpha] = 0$; the case where $[\beta] = 0$ is similar.  Lemma \ref{lemma:sep}
implies that there exists a regular degree $2$ cover $(\Sigma_1,v_1) \rightarrow (\Sigma_0,v_0)$ such that the based
lift $(\alpha_1, v_1)$ is a closed curve satisfying $[a_1] \neq 0$. Let
$((\alpha_1,v_1),\tau_1,(\beta_1,w_1))$ be the lift of $((\alpha_0,v_0),\tau_0,(\beta_0,w_0)$ to $(\Sigma_1,v_1)$. 
If $\beta_1$ is not a closed curve, then $((\alpha_0,v_0),\tau_0,(\beta_0,w_0)$
has only a partially closed lift to $(\Sigma_1,v_1)$ and we are done.  We can therefore assume that $\beta_1$ is closed.
If $[\beta_1] = 0$, then we can apply Lemma \ref{lemma:different} to obtain a regular degree $2$ cover
$(\Sigma_2,v_2) \rightarrow (\Sigma_1,v_1)$ such that $((\alpha_1,v_1),\tau_1,(\beta_1,w_1))$ has only a partially
closed lift to $(\Sigma_2,v_2)$, and we are done.  We can assume therefore that $[\beta_1] \neq 0$.  The desired
tower of covers is then obtained by applying Step 1 to $((\alpha_1,v_1),\tau_1,(\beta_1,w_1))$ and $(\Sigma_1,v_1)$.
\end{proof}


\section{Eliminating $3/28$ of the intersections}
\label{section:basicmove}

This section is devoted to the proof of Lemma \ref{lemma:basicmove}.  The skeleton of the proof is in
\S \ref{section:basicmoveskeleton}.  This skeleton depends on two lemmas which are proven in \S \ref{section:basicmoveresolveisect}
and \S \ref{section:basicmovemakegood}.

\subsection{Skeleton of the proof of Lemma \ref{lemma:basicmove}}
\label{section:basicmoveskeleton}

We begin with some definitions.  Fix a closed surface $\Sigma$ and equip $\Sigma$ with a hyperbolic metric.  Let
$((\alpha,v),\tau,(\beta,w))$ be a curve-arc triple on $\Sigma$. 
A {\em $\beta$-arc} of $\alpha$ is a subarc of $\alpha$ both of whose endpoints
lie in $\beta$.  A $\beta$-arc $\mu$ of $\alpha$ defines in a natural way an element $[\mu]_{\beta}$ in the relative homology group
$\HH_1(\Sigma,\beta;\Field_2)$; we will call $\mu$ a {\em good $\beta$-arc} if $[\mu]_{\beta} \neq 0$ and a {\em bad $\beta$-arc} if $[\mu]_{\beta}=0$.  A
set $\mathcal{A}$ of $\beta$-arcs of $\alpha$ will be called a {\em set of disjoint $\beta$-arcs} of $\alpha$ if for all distinct
$\mu,\mu' \in \mathcal{A}$, we have $\mu \cap \mu' = \emptyset$.
If $\mathcal{A}$ is a set of disjoint $\beta$-arcs of $\alpha$, then define
$$\mathcal{A}^{\text{g}} = \Set{$\mu \in \mathcal{A}$}{$\mu$ is good} \quad \text{and} \quad \mathcal{A}^{\text{b}} = \Set{$\mu \in \mathcal{A}$}{$\mu$ is bad},$$
so $\mathcal{A} = \mathcal{A}^{\text{g}} \sqcup \mathcal{A}^{\text{b}}$.

The following lemma shows that we can eliminate intersections using half of the good $\beta$-arcs of $\alpha$.  The reason for our
notation will become clear during the proof of Lemma \ref{lemma:basicmove} below.

\begin{lemma}
\label{lemma:basicmoveresolveisect}
Let $\Sigma_1$ be a closed surface equipped with a hyperbolic metric and let $((\alpha_1,v_1),\tau_1,(\beta_1,w_1))$ be a curve-arc triple on $\Sigma_1$.  
Assume that $[\alpha_1] = [\beta_1] \neq 0$.  Set
$n_1 = i(\alpha_1,\beta_1)$.  Let $\mathcal{A}_1$ be a set of disjoint $\beta_1$-arcs of $\alpha_1$.   Then there exists a regular
degree $2$ cover $(\Sigma_2,v_2) \rightarrow (\Sigma_1,v_1)$ such that $((\alpha_1,v_1),\tau_1,(\beta_1,w_1))$ has a closed lift
$((\alpha_2,v_2),\tau_2,(\beta_2,w_2))$ with the following property.  
\begin{itemize}
\item Set $n_2 = i(\alpha_2,\beta_2)$.  Then $n_2 \leq n_1 - \frac{1}{2} |\mathcal{A}_1^{\text{g}}|$.
\end{itemize}
\end{lemma}

\noindent
The proof of Lemma \ref{lemma:basicmoveresolveisect} is in \S \ref{section:basicmoveresolveisect}.  It will turn out that
an appropriate ``random'' cover will do the job.  

Unfortunately, it is possible for most of the arcs in a set of disjoint $\beta$-arcs of $\alpha$ to be bad.  The following
lemma shows that we can pass to a cover that makes at least some of the arcs good (or, even better, eliminates some of the intersections).  See
the remark following the lemma for an explanation of the inequalities in its conclusion.

\begin{lemma}
\label{lemma:basicmovemakegood}
Let $\Sigma_0$ be a closed surface equipped with a hyperbolic metric and let $((\alpha_0,v_0),\tau_0,(\beta_0,w_0))$ be a curve-arc triple on $\Sigma_0$.
Assume that $[\alpha_0] = [\beta_0] \neq 0$.  Set
$n_0 = i(\alpha_0,\beta_0)$.  Let $\mathcal{A}_0$ be a set of disjoint $\beta_0$-arcs of $\alpha_0$.   Then there exists a regular
degree $2$ cover $(\Sigma_1,v_1) \rightarrow (\Sigma_0,v_0)$ such that $((\alpha_0,v_0),\tau_0,(\beta_0,w_0))$ has a closed lift
$((\alpha_1,v_1),\tau_1,(\beta_1,w_1))$ with the following property.  
\begin{itemize}
\item Set $n_1 = i(\alpha_1,\beta_1)$.  Then there exists a set $\mathcal{A}_1$ of disjoint $\beta_1$-arcs of $\alpha_1$ and some 
$r \geq 0$ such that
\begin{equation}
\label{eqn:resolve}
|\mathcal{A}_1^{\text{g}}| \geq |\mathcal{A}_0^{\text{g}}| + \frac{3}{7} |\mathcal{A}_0^{\text{b}}| - r \quad \text{and} \quad n_1 \leq n_0 - r.
\end{equation}
\end{itemize}
\end{lemma}

\noindent
The proof of Lemma \ref{lemma:basicmovemakegood} is in \S \ref{section:basicmovemakegood}.  Again, we will see that
a suitable ``random'' cover has the property we seek.

\begin{remark}
The meaning of the inequalities in \eqref{eqn:resolve} is that $\mathcal{A}_1^{\text{g}}$ is made up of pieces corresponding
to $\mathcal{A}_0^{\text{g}}$ (the old good arcs) and $\frac{3}{7}$ of $\mathcal{A}_0^{\text{b}}$ (the old bad arcs), minus some
number $r$ of arcs that correspond to eliminated intersections.
\end{remark}

We now show how to derive Lemma \ref{lemma:basicmove} from Lemmas \ref{lemma:basicmoveresolveisect}--\ref{lemma:basicmovemakegood}.

\begin{proof}[{Proof of Lemma \ref{lemma:basicmove}}]
We begin by recalling the statement.  We are given a closed surface $\Sigma_0$ whose genus is at least $2$ which is equipped
with a hyperbolic metric.  Also, we are given a curve-arc triple $((\alpha_0,v_0),\tau_0,(\beta_0,w_0))$ on $\Sigma_0$.  Setting
$n_0 = i(\alpha_0,\beta_0)$, we are given that $n_0 \geq 2$ and that neither $[\alpha_0]$ nor $[\beta_0]$ vanishes.
Our goal is to prove that for some $\ell$ satisfying $1 \leq \ell \leq 2$, there exists a tower
$$(\Sigma_{\ell},v_{\ell}) \longrightarrow \cdots \longrightarrow (\Sigma_0,v_0)$$
of regular degree $2$ based covers such that one of the following holds.
\begin{itemize}
\item The curve-arc triple $((\alpha_0,v_0),\tau_0,(\beta_0,w_0))$ has
only a partially closed lift to $(\Sigma_{\ell},v_{\ell})$, or
\item The curve-arc triple $((\alpha_0,v_0),\tau_0,(\beta_0,w_0))$ has a closed
lift $((\alpha_{\ell},v_{\ell}),\tau_{\ell},(\beta_{\ell},w_{\ell}))$ to
$(\Sigma_{\ell},v_{\ell})$ such that
$i(\alpha_{\ell}, \beta_{\ell}) \leq \frac{25}{28} n_0$.
\end{itemize}

First, if $[\alpha_0] \neq [\beta_0]$, then Lemma \ref{lemma:different} says that there exists a regular degree $2$ cover
$(\Sigma_1,v_1) \rightarrow (\Sigma_0,v_0)$ such that $((\alpha_0,v_0),\tau_0,(\beta_0,w_0))$ has
only a partially closed lift to $(\Sigma_1,v_1)$, and we are done.  We can assume therefore that $[\alpha_0] = [\beta_0]$.
This implies that $n_0 = i(\alpha_0,\beta_0) \geq 2$ must be an even number.  We can therefore find
a set $\mathcal{A}_0$ of disjoint $\beta_0$-arcs of $\alpha_0$ such that $|\mathcal{A}_0| = n_0/2$.  We then
apply Lemma \ref{lemma:basicmovemakegood} to get a regular degree $2$ cover $(\Sigma_1,v_1) \rightarrow (\Sigma_0,v_0)$ such that
$((\alpha_0,v_0),\tau_0,(\beta_0,w_0))$ has a closed lift $((\alpha_1,v_1),\tau_1,(\beta_1,w_1))$.  The lemma also
gives a set $\mathcal{A}_1$ of disjoint $\beta_1$-arcs of $\alpha_1$ such that for some $r \geq 0$,
$$|\mathcal{A}_1^{\text{g}}| \geq |\mathcal{A}_0^{\text{g}}| + \frac{3}{7} |\mathcal{A}_0^{\text{b}}| - r \quad \text{and} \quad n_1 \leq n_0 - r,$$
where $n_1 = i(\alpha_1, \beta_1)$.

Again, if $[\alpha_1] \neq [\beta_1]$, then Lemma \ref{lemma:different} says that there exists a regular degree $2$ cover
$(\Sigma_2,v_2) \rightarrow (\Sigma_1,v_1)$ such that $((\alpha_1,v_1),\tau_1,(\beta_1,w_1))$ has
only a partially closed lift to $(\Sigma_2,v_2)$, and we are done.  We can assume therefore that $[\alpha_1] = [\beta_1]$.
Lemma \ref{lemma:basicmoveresolveisect} thus gives a regular degree $2$ cover $(\Sigma_2,v_2) \rightarrow (\Sigma_1,v_1)$ such that
$((\alpha_1,v_1),\tau_1,(\beta_1,w_1))$ has a closed lift $((\alpha_2,v_2),\tau_2,(\beta_2,w_2))$ satisfying
\begin{align*}
i\left(\alpha_2,\beta_2\right) &\leq n_1 - \frac{1}{2} |\mathcal{A}_1^{\text{g}}| 
\leq \left(n_0 - r\right) - \frac{1}{2} \left(|\mathcal{A}_0^{\text{g}}| + \frac{3}{7} |\mathcal{A}_0^{\text{b}}| - r\right) \\
&\leq n_0 - \frac{1}{2} \left(|\mathcal{A}_0^{\text{g}}| + \frac{3}{7} |\mathcal{A}_0^{\text{b}}|\right) 
\leq n_0 - \frac{1}{2} \left(\frac{3}{7} |\mathcal{A}_0|\right) \\
&= n_0 - \frac{1}{2} \left(\frac{3}{7} \left(\frac{1}{2} n_0\right)\right) = \frac{25}{28} n_0,
\end{align*}
as desired.
\end{proof}

\subsection{Resolving intersections using good arcs}
\label{section:basicmoveresolveisect}

To prove Lemma \ref{lemma:basicmoveresolveisect}, we need the following lemma.

\begin{lemma}
\label{lemma:random1}
Let $\vec{v}_1,\ldots,\vec{v}_m \in \Field_2^n$ be nonzero vectors (not necessarily distinct).  Then there
exists a linear map $f : \Field_2^n \rightarrow \Field_2$ such that $f(\vec{v}_i) = 1$ for at least half
of the $\vec{v}_i$, i.e.\ such that $\Set{$i$}{$1 \leq i \leq m$, $f(\vec{v}_i) = 1$}$ has
cardinality at least $m/2$.
\end{lemma}
\begin{proof}
Let $\Omega$ be the probability space consisting of all linear maps $\Field_2^n \rightarrow \Field_2$, each given
equal probability.  Let $\mathcal{X} : \Omega \rightarrow \R$ be the random variable that takes $f \in \Omega$ to the
cardinality of the set $\Set{$i$}{$1 \leq i \leq m$, $f(\vec{v}_i) = 1$}$.  We will prove that the expected value $E(\mathcal{X})$ of $\mathcal{X}$
is $m/2$, which clearly implies that there exists {\em some} element $f \in \Omega$ such that $\mathcal{X}(f) \geq m/2$.

To prove the desired claim, for $1 \leq i \leq m$ let $\mathcal{X}_i : \Omega \rightarrow \R$ be the random variable that
takes $f \in \Omega$ to $1$ if $f(\vec{v}_i) = 1$ and to $0$ if $f(\vec{v}_i) = 0$.  Viewing $\vec{v}_i$ as an element of the double dual
$(\Field_2^n)^{\ast \ast}$, the kernel of $\vec{v}_i$ consists of exactly half of the elements of $(\Field_2^n)^{\ast}$.  This
implies $E(\mathcal{X}_i) = 1/2$.  Using linearity of expectation (which, recall, does not require that the random
variables be independent), we get that
$$E\left(\mathcal{X}\right) = E\left(\sum_{i=1}^m \mathcal{X}_i\right) = \sum_{i=1}^m E\left(\mathcal{X}_i\right) = \sum_{i=1}^m \frac{1}{2} = \frac{m}{2},$$
as desired.
\end{proof}

\begin{remark}
One can give a (somewhat more complicated) non-probabilistic proof of Lemma \ref{lemma:random1}; however, we do not know a non-probabilistic proof
of Lemma \ref{lemma:random2} below, which is needed for the proof of Lemma \ref{lemma:basicmovemakegood}.
\end{remark}

\begin{proof}[{Proof of Lemma \ref{lemma:basicmoveresolveisect}}]
We start by recalling the setup.  We are given a closed surface $\Sigma_1$ equipped with a hyperbolic metric and a curve-arc triple
$((\alpha_1,v_1),\tau_1,(\beta_1,w_1))$ on $\Sigma_1$ such that $[\alpha_1] = [\beta_1] \neq 0$.  We are also given
a set $\mathcal{A}_1$ of disjoint $\beta_1$-arcs of $\alpha_1$.  Letting $n_1 = i(\alpha_1, \beta_1)$, our goal
is to construct a regular degree $2$ cover $(\Sigma_2,v_2) \rightarrow (\Sigma_1,v_1)$ such that $((\alpha_1,v_1),\tau_1,(\beta_1,w_1))$ has a closed lift
$((\alpha_2,v_2),\tau_2,(\beta_2,w_2))$ satisfying
$$i(\alpha_2,\beta_2) \leq n_1 - \frac{1}{2} |\mathcal{A}_1^{\text{g}}|.$$

Lemma \ref{lemma:random1} implies that there exists a linear map $\phi : \HH_1(\Sigma_1,\beta_1;\Field_2) \rightarrow \Field_2$ such that
$\Set{$\mu \in \mathcal{A}_1^{\text{g}}$}{$\phi([\mu]_{\beta_1}) = 1$}$ has cardinality at least
$\frac{1}{2} |\mathcal{A}_1^{\text{g}}|$.  Let $\psi : \HH_1(\Sigma_1;\Field_2) \rightarrow \Field_2$ be the composition of
$\phi$ with the natural map $\HH_1(\Sigma_1;\Field_2) \rightarrow \HH_1(\Sigma_1,\beta_1;\Field_2)$ and let
$(\Sigma_2,v_2) \rightarrow (\Sigma_1,v_1)$ be the regular $2$-fold cover
associated to $\psi$.  Since $\psi([\beta_1]) = 0$ and $[\alpha_1] = [\beta_1]$, the curve-arc triple
$((\alpha_1,v_1),\tau_1,(\beta_1,w_1))$ has a closed lift $((\alpha_2,v_2),\tau_2,(\beta_2,w_2))$ to $(\Sigma_2,v_2)$.

It remains to prove that $i(\alpha_2,\beta_2) \leq n_1 - \frac{1}{2} |\mathcal{A}_1^{\text{g}}|$.  Regarding
$S^1$ as the unit circle in $\C$, parametrize $\alpha_1$ via a continuous map $f_1 : (S^1,1) \rightarrow (\Sigma_1,v_1)$.
Define $I_1 = \Set{$\theta \in S^1$}{$f_1(\theta) \in \beta_1$}$.  Since $\alpha_1$ and $\beta_1$ are hyperbolic geodesics,
we have $n_1 = |I_1|$.  Lift $f_1$ to a map $f_2 : (S^1,1) \rightarrow (\Sigma_2,v_2)$ whose image is $\alpha_2$ and
define $I_2 = \Set{$\theta \in S^1$}{$f_2(\theta) \in \beta_2$}$.  Again, since $\alpha_2$ and $\beta_2$ are hyperbolic
geodesics, we must have $i(\alpha_2,\beta_2) = |I_2|$.  Also, by construction we have $I_2 \subset I_1$.  It is enough,
therefore, to prove that $|I_1 \setminus I_2| \geq \frac{1}{2} |\mathcal{A}_1^{\text{g}}|$.

To do this, it is enough to prove that 
\begin{equation}
\label{eqn:i1minusi2}
|I_1 \setminus I_2| \geq |\Set{$\mu \in \mathcal{A}_1^{\text{g}}$}{$\phi([\mu]_{\beta_1}) = 1$}|.
\end{equation}
Consider $\mu \in \mathcal{A}_1^{\text{g}}$ such that $\phi([\mu]_{\beta_1}) = 1$.  There exist $\theta,\theta' \in I_1$ such that
$\mu$ begins at $f_1(\theta)$ and ends at $f_1(\theta')$.  To prove \eqref{eqn:i1minusi2}, it
is enough to prove that at most one of $\theta$ and $\theta'$ lie in $I_2$.  Assume otherwise, so
$\theta, \theta' \in I_2$.  Let $\widetilde{\mu}$ be the oriented arc of $\alpha_2$ beginning at $f_2(\theta)$ and ending
at $f_2(\theta')$ and covering $\mu$.  Also, let $\widetilde{\eta}$ be one of the two oriented arcs of $\beta_2$ beginning at $f_2(\theta')$ and ending
at $f_2(\theta)$. Then $\widetilde{\eta}$ covers an oriented arc $\eta$
of $\beta_1$ beginning at $f_1(\theta')$ and ending at $f_1(\theta)$.  The
closed loop $\mu \cdot \eta$ on $\Sigma_1$ lifts to the closed loop $\widetilde{\mu} \cdot \widetilde{\eta}$ on $\Sigma_2$,
so $\psi([\mu \cdot \eta]) = 0$.  However, we also have
$$\psi([\mu \cdot \eta]) = \phi([\mu]_{\beta_1}) = 1,$$
a contradiction.
\end{proof}

\subsection{Lifting bad arcs to good arcs}
\label{section:basicmovemakegood}

We begin by clarifying the topological nature of bad arcs.

\begin{lemma}
\label{lemma:makegoodsep}
Let $\Sigma$ be a closed surface equipped with a hyperbolic metric and let $((\alpha,v),\tau,(\beta,w))$ be a curve-arc triple on $\Sigma$.  Assume
that $[\alpha] = [\beta] \neq 0$.  Also,
let $\mu$ be a bad $\beta$-arc of $\alpha$ which goes from $p_1 \in \alpha \cap \beta$ to $p_2 \in \alpha \cap \beta$.  
Then there exists an oriented arc $\eta$ of $\beta$ which goes from $p_2$ to $p_1$ with the following properties.
\begin{itemize}
\item The closed curve $\mu \cdot \eta$ is a separating simple closed curve which is not nullhomotopic.
\item The curve $\mu \cdot \eta$ is isotopic to a curve which is disjoint from $\beta$.
\end{itemize}
\end{lemma}
\begin{proof}
There are two arcs $\eta$ and $\eta'$ of $\beta$ going from $p_2$ to $p_1$.  The elements $[\mu \cdot \eta],[\mu \cdot \eta'] \in \HH_1(\Sigma;\Field_2)$
both map to $[\mu]_{\beta} \in \HH_1(\Sigma,\beta;\Field_2)$, which vanishes by assumption.  It follows that $[\mu \cdot \eta]$ and $[\mu \cdot \eta']$
both lie in the kernel of the map $\HH_1(\Sigma;\Field_2) \rightarrow \HH_1(\Sigma,\beta;\Field_2)$, which consists of $0$ and $[\beta]$.  Letting
$\overline{\eta}'$ be $\eta'$ traversed in the opposite direction, we have 
$$[\mu \cdot \eta] - [\mu \cdot \eta'] = [\eta \cdot \overline{\eta}'] = [\beta].$$
We conclude that either $[\mu \cdot \eta]$ or $[\mu \cdot \eta']$ must be $0$; relabeling, we can assume that $[\mu \cdot \eta]=0$.  
The curve $\mu \cdot \eta$ is a simple closed curve which is nullhomologous over $\Field_2$.  It
is not hard to see that this implies that $\mu \cdot \eta$ is in fact nullhomologous over $\Z$, 
i.e.\ that $\mu \cdot \eta$ is a separating simple closed curve (the key point here being
that $\mu \cdot \eta$ is simple).  If
$\mu \cdot \eta$ were nullhomotopic, then it would bound a disc.  This would imply that $\alpha$ and $\beta$ could be homotoped
so as to intersect fewer times, which is impossible since they are hyperbolic geodesics.

It remains to prove that $\mu \cdot \eta$ is isotopic to a curve which is disjoint from $\beta$.  The algebraic intersection number
of $\mu \cdot \eta$ and $\beta$ is $0$ since $\mu \cdot \eta$ is nullhomologous.  If the signs of the intersections of $\alpha$ and $\beta$ 
at $p_1$ and $p_2$ were the same, then $\mu \cdot \eta$ would have algebraic intersection number $1$ with $\beta$
(see Figure \ref{figure:samesignint}.a),
so those signs must be opposite.  The desired result now follows from
Figure \ref{figure:samesignint}.b.
\end{proof}

\Figure{figure:samesignint}{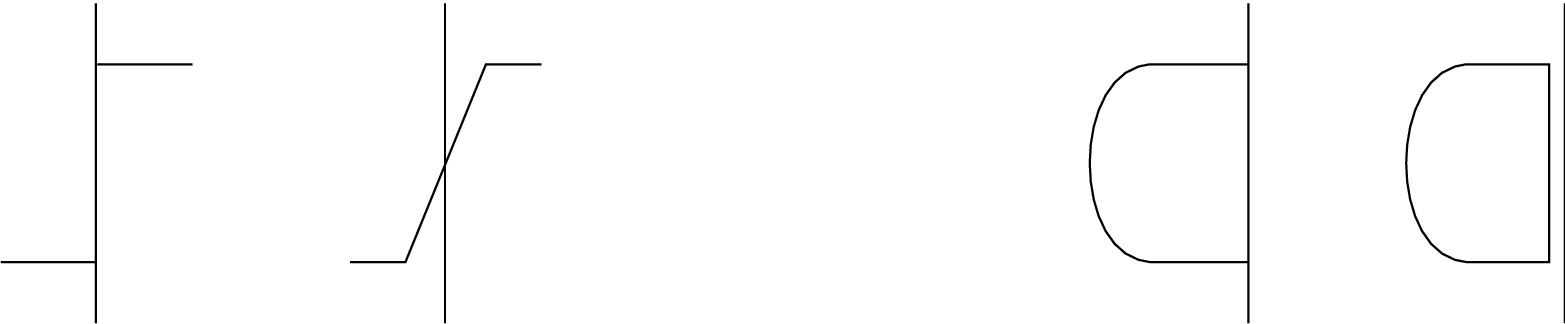}{
a. On the left is a local picture of $\mu$ and $\beta$ when the signs of the intersections
of $\alpha$ and $\beta$ are the same at $p_1$ and $p_2$.  On the right is shown an 
isotopy of $\mu \cdot \eta$ indicating the
algebraic intersection number is $1$ between $\mu \cdot \eta$ and $\beta$.
\CaptionSpace
b. On the left is a local picture of $\mu$ and $\beta$ when the signs of the intersections of $\alpha$
and $\beta$ are different at $p_1$ and $p_2$.  On the right is shown an isotopy of $\mu \cdot \eta$
indicating that $\mu \cdot \eta$ can be made disjoint from $\beta$.}

\noindent
To ``resolve'' bad arcs into good arcs, we will need to lift separating curves to nonseparating curves
as in the following lemma.  Note that if $\delta$ is a simple closed separating curve on $\Sigma$
which divides $\Sigma$ into subsurfaces $S$ and $S'$, then $\HH_1(\Sigma;\Field_2) = \HH_1(S;\Field_2) \oplus \HH_1(S';\Field_2)$.

\begin{lemma}
\label{lemma:badarcresolve}
Let $\Sigma$ be a closed surface and let $\delta$ be a simple closed separating curve on $\Sigma$ which divides $\Sigma$
into subsurfaces $S$ and $S'$.  Let $\phi : \HH_1(\Sigma;\Field_2) \rightarrow \Field_2$ be a linear map such that
$\phi|_{\HH_1(S;\Field_2)} \neq 0$ and $\phi|_{\HH_1(S';\Field_2)} \neq 0$.  Let $\widetilde{\Sigma} \rightarrow \Sigma$ be
the regular $2$-fold cover associated to $\phi$.  Finally, let $\widetilde{\delta}$ be a component of the preimage of $\delta$
in $\widetilde{\Sigma}$.  Then $[\widetilde{\delta}] \neq 0$.
\end{lemma}
\begin{proof}
Let $\widetilde{S}$ and $\widetilde{S}'$ be the preimages in $\widetilde{\Sigma}$ of $S$ and $S'$, respectively.  The maps
$\widetilde{S} \rightarrow S$ and $\widetilde{S}' \rightarrow S'$ are $2$-fold covering maps.  Moreover (and this is the key
observation), the assumption that $\phi|_{\HH_1(S;\Field_2)} \neq 0$ and $\phi|_{\HH_1(S';\Field_2)} \neq 0$ implies that both $\widetilde{S}$
and $\widetilde{S}'$ are connected.  Since $\delta$ is a separating curve, we have $[\delta] = 0$ and thus $\phi([\delta])=0$, 
which implies that the preimage of $\delta$ in $\widetilde{\Sigma}$ has two components.  These two components are the boundary
components of $\widetilde{S}$ and $\widetilde{S}'$, so we deduce that both $\widetilde{S}$ and $\widetilde{S}'$ have two boundary
components.  The surface $\widetilde{\Sigma}$ is obtained by gluing the boundary components of $\widetilde{S}$ to the boundary
components of $\widetilde{S}'$.  Consequently, $\widetilde{\delta}$ (which is one of those boundary components) does not
separate $\widetilde{\Sigma}$, so $[\widetilde{\delta}] \neq 0$.
\end{proof}


\noindent
We thus need to construct linear maps $\HH_1(\Sigma;\Field_2) \rightarrow \Field_2$ which are nontrivial on a large number
of splittings of $\HH_1(\Sigma;\Field_2)$.  We will do this with the following lemma, which
plays the same role in the proof of Lemma \ref{lemma:basicmovemakegood} that Lemma \ref{lemma:random1}
played in the proof of Lemma \ref{lemma:basicmoveresolveisect}.

\begin{lemma}
\label{lemma:random2}
Fix some $n \geq 3$.  For $1 \leq i \leq m$, let $V_i,W_i \subset \Field_2^n$ be nontrivial subspaces such that $\Field_2^n = V_i \oplus W_i$.
Then there exists a linear map $f : \Field_2^n \rightarrow \Field_2$ such that
$\Set{$i$}{$1 \leq i \leq m$, $f|_{V_i} \neq 0$, $f|_{W_i} \neq 0$}$ has cardinality at least $\frac{3}{7} m$.
\end{lemma}
\begin{proof}
Let $\Omega$ be the probability space consisting of all {\em nonzero} linear maps $\Field_2^n \rightarrow \Field_2$, each given
equal probability.  Let $\mathcal{X} : \Omega \rightarrow \R$ be the random variable that takes $f \in \Omega$ to the
cardinality of the set $\Set{$i$}{$1 \leq i \leq m$, $f|_{V_i} \neq 0$, $f|_{W_i} \neq 0$}$.  
We will prove that the expected value $E(\mathcal{X})$ of $\mathcal{X}$
is at least $\frac{3}{7} m$, which clearly implies that there exists {\em some} element $f \in \Omega$ such that $\mathcal{X}(f) \geq \frac{3}{7} m$.

To prove the desired claim, for $1 \leq i \leq m$ let $\mathcal{X}_i : \Omega \rightarrow \R$ be the random variable that
takes $f \in \Omega$ to $1$ if $f|_{V_i},f|_{W_i} \neq 0$ and to $0$ if $f|_{V_i} = 0$ or  $f|_{W_i} = 0$.  Using linearity of expectation,
we have
$$E\left(\mathcal{X}\right) = E\left(\sum_{i=1}^m \mathcal{X}_i\right) = \sum_{i=1}^m E\left(\mathcal{X}_i\right).$$
Fixing some $1 \leq i \leq m$, it is therefore enough to prove that $E(\mathcal{X}_i) \geq 3/7$, i.e.\ 
to show that the probability that a random nonzero linear map $f : \Field_2^n \rightarrow \Field_2$ satisfies 
$f|_{V_i},f|_{W_i} \neq 0$ is at least $3/7$.

Set $a = \dim(V_i)$, so $1 \leq a \leq n-1$ and $\dim(W_i) = n-a$.  There are $2^a-1$ (resp.\ $2^{n-a}-1$) nonzero
linear maps $V_i \rightarrow \Field_2$ (resp.\ $W_i \rightarrow \Field_2$).  This implies that there are $(2^a-1)(2^{n-a}-1)$ linear
maps $f : \Field_2^n \rightarrow \Field_2$ such that $f|_{V_i},f|_{W_i} \neq 0$.  Since there are $2^n-1$ nonzero linear
maps $\Field_2^n \rightarrow \Field_2$, we deduce that the probability that a random nonzero linear map $f : \Field_2^n \rightarrow \Field_2$ satisfies 
$f|_{V_i},f|_{W_i} \neq 0$ is $\frac{(2^a-1)(2^{n-a}-1)}{2^n-1}$.  Lemma \ref{lemma:calc} below says that this is at least $3/7$, as desired.
\end{proof}

\begin{lemma}
\label{lemma:calc}
Set $D = \Set{$(a,n)$}{$n \geq 3$ and $1 \leq a \leq n-1$}$, and define $\zeta : D \rightarrow \R$ via the formula
$\zeta(a,n) = \frac{(2^a-1)(2^{n-a}-1)}{2^n-1}$.  Then $\zeta$ is bounded below by $\zeta(1,3) = 3/7$.
\end{lemma}
\begin{proof}
For a fixed $n_0 \geq 3$, it is easy to see that the function $a \mapsto \zeta(a,n_0)$ defined on the domain $1 \leq a \leq n_0-1$
attains its global minima at the boundary points $a=1$ and $a=n_0-1$.  Moreover, $\zeta(1,n_0) = \zeta(n_0-1,n_0)$.  Finally,
the function $n \mapsto \zeta(1,n)$ defined on the domain $n \geq 3$ achieves its global minimum at the point $n=3$.  The lemma follows.
\end{proof}

\noindent
We are finally in a position to prove Lemma \ref{lemma:basicmovemakegood}.

\begin{proof}[{Proof of Lemma \ref{lemma:basicmovemakegood}}]
We begin by recalling the setup.  We are given a closed surface $\Sigma_0$ equipped with a hyperbolic
metric and a curve-arc triple $((\alpha_0,v_0),\tau_0,(\beta_0,w_0))$ on $\Sigma_0$ such that
$[\alpha_0] = [\beta_0] \neq 0$.  We are also given a set $\mathcal{A}_0$ of disjoint $\beta_0$-arcs of $\alpha_0$.
Our goal is to construct a regular degree $2$ cover 
$(\Sigma_1,v_1) \rightarrow (\Sigma_0,v_0)$ such that $((\alpha_0,v_0),\tau_0,(\beta_0,w_0))$ has a closed lift
$((\alpha_1,v_1),\tau_1,(\beta_1,w_1))$ and there exists a set 
$\mathcal{A}_1$ of disjoint $\beta_1$-arcs of $\alpha_1$ and some $r \geq 0$ satisfying
$$|\mathcal{A}_1^{\text{g}}| \geq |\mathcal{A}_0^{\text{g}}| + \frac{3}{7} |\mathcal{A}_0^{\text{b}}| - r \quad \text{and} \quad n_1 \leq n_0 - r,$$
where $n_1 = i(\alpha_1,\beta_1)$.

Write $\mathcal{A}_0^{\text{b}} = \{\mu_1,\ldots,\mu_{|\mathcal{A}_0^{\text{b}}|}\}$.  Using Lemma \ref{lemma:makegoodsep}, we
obtain oriented arcs $\eta_1,\ldots,\eta_{|\mathcal{A}_0^{\text{b}}|}$ of $\beta_0$ such that the following hold for $1 \leq i \leq |\mathcal{A}_0^{\text{b}}|$.
\begin{itemize}
\item The arc $\eta_i$ begins at the endpoint of $\mu_i$ and ends at the beginning point of $\mu_i$.
\item The curve $\mu_i \cdot \eta_i$ is a simple closed separating curve isotopic to a curve disjoint from $\beta_0$.
\end{itemize}
For $1 \leq i \leq |\mathcal{A}_0^{\text{b}}|$, let $S_i$ and $S_i'$ be the two subsurfaces into which $\Sigma_0$ is divided by $\mu_i \cdot \eta_i$, ordered
so that $\beta_0$ is isotopic into $S_i'$.  We thus have $\HH_1(\Sigma_0;\Field_2) = \HH_1(S_i;\Field_2) \oplus \HH_1(S_i';\Field_2)$ and
$[\beta_0] \in \HH_1(S_i';\Field_2)$.  Let $X$ be the quotient of $\HH_1(\Sigma_0;\Field_2)$ by 
the span of $[\beta_0]$, so $X \cong \HH_1(\Sigma_0,\beta;\Field_2)$.  Also,
let $V_i$ and $W_i$ be the projections to $X$ of $\HH_1(S_i;\Field_2)$ and $\HH_1(S_i';\Field_2)$, respectively.  Thus $X = V_i \oplus W_i$.
Lemma \ref{lemma:random2} implies that there exists a linear map $\phi : X \rightarrow \Field_2$ such that the set
$\Set{$i$}{$1 \leq i \leq |\mathcal{A}_0^{\text{b}}|$, $\phi|_{V_i} \neq 0$, $\phi|_{W_i} \neq 0$}$ has cardinality at least
$\frac{3}{7} |\mathcal{A}_0^{\text{b}}|$.  

Let $\psi : \HH_1(\Sigma_0;\Field_2) \rightarrow \Field_2$ be the composition of $\phi$ with the
projection
$\HH_1(\Sigma_0;\Field_2) \rightarrow X$ and let $(\Sigma_1,v_1) \rightarrow (\Sigma_0,v_0)$ be the regular
$2$-fold cover associated to $\psi$.  Since $\psi([\beta_0]) = 0$ and $[\alpha_0] = [\beta_0]$, the
curve-arc triple $((\alpha_0,v_0),\tau_0,(\beta_0,w_0))$ has a closed lift $((\alpha_1,v_1),\tau_1,(\beta_1,w_1))$
to $(\Sigma_1,v_1)$.  Define $n_1 = i(\alpha_1, \beta_1)$.
Each arc in $\mathcal{A}_0$ lifts to an arc of $\alpha_1$; let $\mathcal{B}_1$ be the
set of these arcs.  Not all of the elements of $\mathcal{B}_1$ are $\beta_1$-arcs.  Let
$\mathcal{A}_1 = \Set{$\mu \in \mathcal{B}_1$}{$\mu$ is a $\beta_1$-arc}$.  Finally, define $r = |\mathcal{B}_1 \setminus \mathcal{A}_1|$.

We have constructed all of the objects claimed by the lemma.  It remains to verify the inequalities from its
conclusion.  Regarding $S^1$ as the unit circle in $\C$, parametrize $\alpha_0$ via a continuous map
$f_0 : (S^1,1) \rightarrow (\Sigma_0,v_0)$.  Define $I_0 = \Set{$\theta \in S^1$}{$f_0(\theta) \in \beta_0$}$.  Since
$\alpha_0$ and $\beta_0$ are hyperbolic geodesics, we have $n_0 = |I_0|$.  Lift $f_0$ to a map $f_1 : (S^1,1) \rightarrow (\Sigma_1,v_1)$
whose image is $\alpha_1$ and define $I_1 = \Set{$\theta \in S^1$}{$f_1(\theta) \in \beta_1$}$.  Again, since $\alpha_1$
and $\beta_1$ are hyperbolic geodesics, we must have $n_1 = |I_1|$.  Also, by construction we have $I_1 \subset I_0$.  The
desired inequality $n_1 \leq n_0 - r$ is thus equivalent to the assertion that $|I_0 \setminus I_1| \geq r$.  This
follows immediately from the fact that the elements of $\mathcal{B}_1 \setminus \mathcal{A}_1$ are exactly the arcs
of $\mathcal{B}_1$ at least one of whose endpoints is $f_1(\theta)$ for some $\theta \in I_0 \setminus I_1$.

All that is left to do is to prove that 
\begin{equation}
\label{eqn:lastthing}
|\mathcal{A}_1^{\text{g}}| \geq |\mathcal{A}_0^{\text{g}}| + \frac{3}{7} |\mathcal{A}_0^{\text{b}}| - r.
\end{equation}
We first define $\mathcal{B}_1^{\text{g}}$ and $\mathcal{B}_1^{\text{b}}$ to be the subsets of $\mathcal{B}_1$ that
are lifts of elements of $\mathcal{A}_0^{\text{g}}$ and $\mathcal{A}_0^{\text{b}}$, respectively.  Next, define
$r_{\text{g}} = |\mathcal{B}_1^{\text{g}} \setminus \mathcal{A}_1|$ and $r_{\text{b}} = |\mathcal{B}_1^{\text{b}} \setminus \mathcal{A}_1|$,
so $r = r_{\text{g}} + r_{\text{b}}$.  We have
\begin{equation}
\label{eqn:last1}
|\mathcal{A}_1^{\text{g}}| = |\mathcal{B}_1^{\text{g}} \cap \mathcal{A}_1^{\text{g}}| + |\mathcal{B}_1^{\text{b}} \cap \mathcal{A}_1^{\text{g}}|.
\end{equation}
Moreover, it is clear from the definitions that $\mathcal{B}_1^{\text{g}} \cap \mathcal{A}_1^{\text{g}} = \mathcal{B}_1^{\text{g}} \cap \mathcal{A}_1$, so
\begin{equation}
\label{eqn:last2}
|\mathcal{B}_1^{\text{g}} \cap \mathcal{A}_1^{\text{g}}| = |\mathcal{B}_1^{\text{g}} \cap \mathcal{A}_1| = |\mathcal{A}_0^{\text{g}}| - r_{\text{g}}.
\end{equation}
Combining \eqref{eqn:last1} and \eqref{eqn:last2}, we see that \eqref{eqn:lastthing} is equivalent to the assertion that
\begin{equation}
\label{eqn:last3}
|\mathcal{B}_1^{\text{b}} \cap \mathcal{A}_1^{\text{g}}| \geq \frac{3}{7} |\mathcal{A}_0^{\text{b}}| - r_{\text{b}},
\end{equation}
which we will now prove.

Recall that we have enumerated $\mathcal{A}_0^{\text{b}}$ as $\{\mu_1,\ldots,\mu_{|\mathcal{A}_0^{\text{b}}|}\}$.  For $1 \leq i \leq |\mathcal{A}_0^{\text{b}}|$, 
let the lift
of $\mu_i$ to $\alpha_1$ be $\widetilde{\mu}_i \in \mathcal{B}_1^{\text{b}}$.  We know that exactly $r_{\text{b}}$ elements
of $\{\widetilde{\mu}_1,\ldots,\widetilde{\mu}_{|\mathcal{A}_0^{\text{b}}|}\}$ do {\em not} lie in $\mathcal{A}_1$.  Also, we know that
the set $\Set{$i$}{$1 \leq i \leq |\mathcal{A}_0^{\text{b}}|$, $\phi|_{V_i} \neq 0$, $\phi|_{W_i} \neq 0$}$ has cardinality
at least $\frac{3}{7} |\mathcal{A}_0^{\text{b}}|$.  We conclude that the set
$$\mathcal{C} := \Set{$\widetilde{\mu}_i$}{$1 \leq i \leq |\mathcal{A}_0^{\text{b}}|$, $\phi|_{V_i} \neq 0$, $\phi|_{W_i} \neq 0$, $\widetilde{\mu}_i \in \mathcal{A}_1$} \subset \mathcal{B}_1^{\text{b}}$$
has cardinality at least $\frac{3}{7} |\mathcal{A}_0^{\text{b}}| - r_{\text{b}}$.  To prove \eqref{eqn:last3}, therefore,
it is enough to prove that $\mathcal{C} \subset \mathcal{A}_1^{\text{g}}$.

Consider some $\widetilde{\mu}_i \in \mathcal{C}$.  The endpoints of $\widetilde{\mu}_i$ lie in $\beta_1$ by construction. Let $\widetilde{\eta}_i$ be the lift
of $\eta_i$ to $\beta_1$ where $\eta_i$ is the curve constructed in the second paragraph.  Lemma \ref{lemma:badarcresolve} 
implies that $[\widetilde{\mu}_i \cdot \widetilde{\eta}_i] \neq 0$.  Also,
we know that $[\widetilde{\mu}_i \cdot \widetilde{\eta}_i]$ projects to $[\mu_i \cdot \eta_i] = 0$ in $\HH_1(\Sigma_0;\Field_2)$, and
$[\beta_1]$ projects to $[\beta_0] \neq 0$.  We deduce that $[\widetilde{\mu}_i \cdot \widetilde{\eta}_i] \neq [\beta_1]$.
Since
$[\widetilde{\mu}_i]_{\beta_1}$ is the projection of $[\widetilde{\mu}_i \cdot \widetilde{\eta}_i]$ to $\HH_1(\Sigma_1,\beta_1;\Field_2)$,
we conclude that $[\widetilde{\mu}_i]_{\beta_1} \neq 0$, i.e.\ that $\widetilde{\mu}_i$ is a good $\beta_1$-arc of $\alpha_1$ and thus
an element of $\mathcal{A}_1^{\text{g}}$, as desired.
\end{proof}

\appendix

\section{Appendix : Surfaces with boundary} 
\label{appendix:nonclosed}

In this appendix, we prove the following theorem.  It generalizes Theorem \ref{maintheorem:isect}
to surfaces with boundary.

\begin{maintheorem} 
\label{theorem:isectbdry}
Let $\Sigma$ be a compact orientable surface of genus $g$ with $b \geq 1$ boundary
components.  Assume that $2g+b \geq 3$.
Let $\alpha$ and $\beta$ be nonisotopic oriented simple closed curves on $\Sigma$.  Assume
that for some $d \geq 7$ the conjugacy classes in $\pi_1(\Sigma) / \gamma_d(\pi_1(\Sigma))$
associated to $\alpha$ and $\beta$ are the same.  Then
$$i(\alpha,\beta) \geq \left(\frac{d+2}{2}\right)^c,$$
where $c = \frac{\ln(28/25)}{\ln(4)}$.
\end{maintheorem}

\begin{remark}
In the exceptional cases where $2g+b < 3$, the theorem would be vacuous.
\end{remark}

\noindent
We begin with an elementary lemma.

\begin{lemma} \label{lemma:embedinclosed}
Let $\Sigma$ be an orientable compact surface of genus $g \geq 0$ with $b \geq 1$ boundary components.
Assume that $2g+b \geq 3$.
Let $\alpha$ and $\beta$ be nonisotopic simple closed curves on $\Sigma$.
Then there is a closed surface $\Sigma'$ of genus at least $2$ and an embedding 
$f: \Sigma \hookrightarrow \Sigma'$ such that $f(\alpha)$ and $f(\beta)$ are nonisotopic on $\Sigma'$.
\end{lemma}
\begin{proof}
Let $\Sigma'$ be the double of $\Sigma$, i.e.\ let 
$\Sigma' = (\Sigma \times \{0\} \cup \Sigma \times \{1\})/ \sim$,
where we identify $(x, 0) \sim (x, 1)$ for all $x \in \partial \Sigma$. 
Then $\Sigma'$ is a closed orientable surface of genus at least $2$.  Moreover, 
there is a retraction $r: \Sigma' \to \Sigma$ defined by $r( \overline{(x, i)}) = x$.

Suppose for the sake of contradiction that $f(\alpha)$ and $f(\beta)$ were isotopic curves on $\Sigma'$. 
Then $\alpha = r(f(\alpha))$ is homotopic to $\beta = r(f(\beta))$ on $\Sigma$.  Since homotopic
simple closed curves on a surface are isotopic (see, e.g., \cite[Proposition 1.10]{FarbMargalitPrimer}),
this is a contradiction.
\end{proof}

\begin{proof}[Proof of Theorem \ref{theorem:isectbdry}]
Let $\Sigma'$ and $f$ be as in Lemma \ref{lemma:embedinclosed} and let 
$f_*$ be the induced homomorphism on $\pi_1$.
Then $f_*(\gamma_d(\pi_1(\Sigma))) \subseteq \gamma_d(\pi_1(\Sigma'))$, and so there is an induced map
$\pi_1(\Sigma)/\gamma_d(\pi_1(\Sigma)) \to \pi_1(\Sigma')/\gamma_d(\pi_1(\Sigma'))$. This implies
that $f(\alpha)$ and $f(\beta)$ have the same conjugacy class in $\pi_1(\Sigma')/\gamma_d(\pi_1(\Sigma'))$.
Moreover, it is clear that $i(\alpha, \beta) \geq i(f(\alpha), f(\beta))$, and from the lemma $f(\alpha)$
and $f(\beta)$ are nonisotopic. By Theorem \ref{maintheorem:isect}, 
\[i(\alpha, \beta) \geq i(f(\alpha), f(\beta)) \geq \left( \frac{d+2}{2}\right)^c.\qedhere\]
\end{proof}

\noindent
\begin{tabular*}{\linewidth}[t]{@{}p{\widthof{E-mail: {\tt justinmalestein@math.huji.ac.il}}+0.75in}@{}p{\linewidth - \widthof{E-mail: {\tt justinmalestein@math.huji.ac.il}} - 0.75in}@{}}
{\raggedright
Justin Malestein\\
Einstein Institute of Mathematics \\
Edmond J. Safra Campus, Givat Ram \\
The Hebrew University of Jerusalem \\
Jerusalem, 91904, Israel\\
E-mail: {\tt justinmalestein@math.huji.ac.il}}
&
{\raggedright
Andrew Putman\\
Department of Mathematics\\
Rice University, MS 136 \\
6100 Main St.\\
Houston, TX 77005\\
E-mail: {\tt andyp@rice.edu}}
\end{tabular*}


\begin{thebibliography}{}
\begin{footnotesize}
\setlength{\itemsep}{1pt}

\bibitem{ArnouxYoccoz}
P. Arnoux\ and\ J.-C. Yoccoz, Construction de diff\'eomorphismes pseudo-Anosov, C. R. Acad. Sci. Paris S\'er. I Math. {\bf 292} (1981), no.~1, 75--78. 

\bibitem{BestvinaBuxMargalit}
M. Bestvina, K.-U. Bux\ and\ D. Margalit, The dimension of the Torelli group, J. Amer. Math. Soc. {\bf 23} (2010), no.~1, 61--105.

\bibitem{DummitFoote}
D. Dummit\ and\ R. Foote, {\it Abstract algebra}, third edition, John Wiley and Sons, Inc., Hoboken, New Jersey, 2004.

\bibitem{FarbLeiningerMargalit}
B. Farb, C. J. Leininger\ and\ D. Margalit, The lower central series and pseudo-Anosov dilatations, Amer. J. Math. {\bf 130} (2008), no.~3, 799--827. 

\bibitem{FarbMargalitPrimer}
B. Farb\ and\ D. Margalit, {\it A primer on mapping class groups}, Princeton Mathematical Series, 49, Princeton Univ. Press, Princeton, NJ, 2012. 

\bibitem{FLP}
A. Fathi, F. Laudenbach, V. Poenaru, et al., {\it Travaux de Thurston sur les surfaces}, Ast\'erisque, 66-67, Soc. Math. France, Paris, 1979.

\bibitem{GaroufalidisLevine}
S. Garoufalidis\ and\ J. Levine, Finite type $3$-manifold invariants and the structure of the Torelli group. I, Invent. Math. {\bf 131} (1998), no.~3, 541--594.

\bibitem{Ivanov}
N. V. Ivanov, Zap. Nauchn. Sem. Leningrad. Otdel. Mat. Inst. Steklov. (LOMI) {\bf 167} (1988), Issled. Topol. 6, 111--116, 191; translation in J. Soviet Math. {\bf 52} (1990), no.~1, 2819--2822.

\bibitem{JohnsonAbel}
D. Johnson, An abelian quotient of the mapping class group ${\cal I}\sb{g}$, Math. Ann. {\bf 249} (1980), no.~3, 225--242. 

\bibitem{JohnsonSurvey}
D. Johnson, A survey of the Torelli group, in {\it Low-dimensional topology (San Francisco, Calif., 1981)}, 165--179, Contemp. Math., 20 Amer. Math. Soc., Providence, RI.

\bibitem{KawazumiKuno}
N. Kawazumi\ and\ Y. Kuno, The logarithms of Dehn twists, to appear in Quantum Topology.

\bibitem{MalesteinPutman}
J. Malestein\ and\ A. Putman, On the self-intersections of curves deep in the lower central series of a surface group, Geom. Dedicata {\bf 149} (2010), 73--84. 

\bibitem{Matsumoto}
M. Matsumoto, Arithmetic mapping class groups, to appear in Park City Mathematics Series.

\bibitem{PennerDil}
R. C. Penner, Bounds on least dilatations, Proc. Amer. Math. Soc. {\bf 113} (1991), no.~2, 443--450.


\end{footnotesize}
\end{thebibliography}
\end{document}